\documentclass{icm2010}

\title[Non-asymptotic theory of random matrices]{Non-asymptotic theory of random \\
    matrices: extreme singular values}

\author[M. Rudelson, R. Vershynin]
{Mark Rudelson\thanks{Partially supported by NSF grant DMS FRG 0652684},
Roman Vershynin\thanks{Partially supported by NSF grant DMS FRG 0918623}}
\contact[rudelsonm@missouri.edu]
{Department of Mathematics\\
University of Missouri-Columbia\\
Columbia, Missouri, U.S.A.}
\contact[romanv@umich.edu]
{Department of Mathematics\\
University of Michigan\\
Ann Arbor, Michigan, U.S.A.}

\newtheorem{theorem}{Theorem}[section]
\newtheorem{proposition}[theorem]{Proposition}
\newtheorem{corollary}[theorem]{Corollary}

\numberwithin{equation}{section}

\theoremstyle{definition}
\newtheorem{definition}[theorem]{Definition}

\theoremstyle{remark}
\newtheorem*{remark}{Remark}

\DeclareMathOperator{\pdf}{pdf}
\DeclareMathOperator{\supp}{supp}
\DeclareMathOperator{\Span}{span}
\DeclareMathOperator{\dist}{dist}
\DeclareMathOperator{\lcd}{lcd}
\DeclareMathOperator{\Median}{Median}
\DeclareMathOperator*{\Ave}{Ave}

\def \R {\mathbb{R}}
\def \C {\mathbb{C}}
\def \Z {\mathbb{Z}}
\def \E {\mathbb{E}}
\def \P {\mathbb{P}}

\def \LL {\mathcal{L}}

\def \NN {\mathcal{N}}
\def \a {\alpha}
\def \b {\beta}

\def \e {\varepsilon}
\def \d {\delta}
\def \l {\lambda}

\def \< {\langle}
\def \> {\rangle}

\def \smin {s_{\min}}
\def \smax {s_{\max}}
\def \Sparse {{\mathit{Sparse}}}
\def \Spread {{\mathit{Spread}}}

\begin{document}

\begin{abstract}
The classical random matrix theory is mostly focused on asymptotic spectral properties
of random matrices as their dimensions grow to infinity. At the same time many recent
applications from convex geometry to functional analysis to information theory
operate with random matrices in fixed dimensions.
This survey addresses the non-asymptotic theory of extreme singular values
of random matrices with independent entries. We focus on recently developed
geometric methods for estimating the hard edge of random matrices (the smallest singular value).
\end{abstract}



\maketitle

\section{Asymptotic and non-asymptotic problems on random matrices}     \label{s: asymptotic-nonasymptotic}

Since its inception, random matrix theory has been
mostly preoccupied with asymptotic properties of random matrices as
their dimensions grow to infinity. A foundational example of this
nature is {\em Wigner's semicircle law} \cite{Wigner}. It applies to
a family of $n \times n$ symmetric matrices $A_n$ whose entries on
and above the diagonal are independent standard normal random
variables. In the limit as the dimension $n$ grows to infinity, the
spectrum of the normalized matrices $\frac{1}{\sqrt{n}} A_n$ is
distributed according to the semicircle law with density
$\frac{1}{2\pi} \sqrt{4 - x^2}$ supported on the interval $[-2,2]$.
Precisely, if we denote by $S_n (z)$ the number of eigenvalues of
$\frac{1}{\sqrt{n}} A_n$ that are smaller than $z$, then for every
$z \in \R$ one has
$$
\frac{S_n (z)}{n} \to \frac{1}{2\pi} \int_{-\infty}^z (4 - x^2)_+^{1/2} \; dx
\quad \text{almost surely as } n \to \infty.
$$

In a similar way, {\em Marchenko-Pastur law} \cite{Marchenko-Pastur}
governs the limiting spectrum of $n \times n$ Wishart matrices $W_{N,n} = A^* A$, where $A=A_{N,n}$
is an $N \times n$ {\em random Gaussian matrix}  whose entries are independent
standard normal random variables.
As the dimensions $N,n$ grow to infinity
while the aspect ratio $n/N$ converges to a non-random number $y \in (0,1]$,
the spectrum of the normalized Wishart matrices $\frac{1}{N} W_{N,n}$
is distributed according to the Marchenko-Pastur law with density
$\frac{1}{2 \pi x y} \sqrt{(b-x)(x-a)}$ supported on  $[a,b]$
where $a = (1-\sqrt{y})^2$, $b = (1+\sqrt{y})^2$. The meaning
of the convergence is similar to the one in Wigner's semicircle law.

It is widely believed that phenomena typically observed in asymptotic random matrix theory
are {\em universal}, that is independent of the particular distribution of the entries of random matrices.
By analogy with classical probability, when we work with independent standard normal
random variables $Z_i$, we know that their normalized sum $S_n = \frac{1}{\sqrt{n}} \sum_{i=1}^n Z_i$
is again a standard normal random variable. This simple but useful fact becomes significantly more
useful when we learn that it is asymptotically universal.
Indeed, The Central Limit Theorem states that if instead of normal distribution $Z_i$
have general identical distribution with zero mean and unit variance,
the normalized sum $S_n$ will still converge (in distribution) to the standard normal random variable
as $n \to \infty$.
In random matrix theory, universality has been established for many results.
In particular, Wigner's semicircle law and Marchenko-Pastur law are known to be universal --
like the Central Limit Theorem,
they hold for arbitrary distribution of entries with zero mean and unit variance
(see \cite{Pastur, BS} for semi-circle law and \cite{Watcher, Bai 99} for Marchenko-Pastur law).

Asymptotic random matrix theory offers remarkably precise predictions
as dimension grows to infinity. At the same time, sharpness at infinity is
often counterweighted by lack of understanding of what happens in finite dimensions.
Let us briefly return to the analogy with the Central Limit Theorem.
One often needs to estimate the sum of independent random variables $S_n$ with fixed
number of terms $n$ rather than in the limit $n \to \infty$. In this situation one may turn to
Berry-Esseen's theorem which quantifies deviations of the distribution of $S_n$ from that of the standard
normal random variable $Z$. In particular, if $\E |Z_1|^3 = M < \infty$ then
\begin{equation}                            \label{Berry-Esseen}
|\P (S_n \le z) - \P (Z \le z)| \le \frac{C}{1+|z|^3} \cdot \frac{M}{\sqrt{n}},
\quad z \in \R,
\end{equation}
where $C$ is an absolute constant \cite{Berry, Esseen}.
Notwithstanding the optimality of Berry-Esseen inequality \eqref{Berry-Esseen},
one can still hope for something better than the polynomial bound on the probability,
especially in view of the super-exponential tail of the limiting normal distribution:
$\P (|Z| > z) \lesssim \exp(-z^2/2)$. Better estimates would indeed emerge in the form of
{\em exponential deviation inequalities} \cite{Petrov, Ledoux concentration},
but this would only happen when we
drop explicit comparisons to the limiting distribution and study the tails of $S_n$ by themselves.
In the simplest case, when $Z_i$ are i.i.d. mean zero random variables bounded in absolute value by $1$,
one has
\begin{equation}                \label{deviation}
\P (|S_n| > z) \le 2 \exp(- c z^2),
\quad z \ge 0,
\end{equation}
where $c$ is a positive absolute constant.
Such exponential deviation inequalities, which are extremely useful in a number of applications,
are non-asymptotic results
whose asymptotic prototype is the Central Limit Theorem.

A similar non-asymptotic viewpoint can be adopted in random matrix theory.
One would then study spectral properties of random matrices of fixed dimensions.
Non-asymptotic results on random matrices are in demand in a number of today's applications that operate
in high but fixed dimensions. This usually happens in statistics where one analyzes data sets
with a large but fixed number of parameters, in geometric functional analysis where one
works with random operators on finite-dimensional spaces (whose dimensions are large
but fixed), in signal processing where the signal is randomly sampled in many but fixed number of points,
and in various other areas of science and engineering.

This survey is mainly focused on the non-asymptotic theory of the extreme singular
values of random matrices (equivalently, the extreme eigenvalues
of sample covariance matrices) where significant progress was made recently.
In Section~\ref{s: extreme} we review estimates on the largest singular value (the soft edge).
The more difficult problem of estimating the smallest singular value (the hard edge) is discussed
in Section~\ref{s: smallest}, and its connection with the Littlewood-Offord problem in additive
combinatorics is the content of Section~\ref{s: Littlewood-Offord}.
In Section~\ref{s: applications} we discuss several applications of non-asymptotic results
to the circular law in asymptotic random matrix theory,
to restricted isometries in compressed sensing,
and to Kashin's subspaces in geometric functional analysis.

This paper is by no means a comprehensive survey of the area but rather a tutorial.
Sketches of some arguments are included in order to give the reader a flavor of non-asymptotic
methods. To do this more effectively, we state most theorems in simplified form (e.g. always over the field $\R$);
the reader will find full statements in the original papers.
Also, we had to completely omit several important directions.
These include random symmetric matrices which were the subject of the recent
survey by Ledoux~\cite{Ledoux-Tracy-Widom} and random matrices with independent columns,
see in particular \cite{ALPT, V covariance}. The reader is also encouraged to look at the comprehensive
survey \cite{DS} on some geometric aspects of random matrix theory.

\section{Extreme singular values}                   \label{s: extreme}

\paragraph{Geometric nature of extreme singular values}
The non-asymptotic viewpoint in random matrix theory is largely motivated
by geometric problems in high dimensional Euclidean spaces.
When we view an $N \times n$ matrix $A$ as a linear operator
$\R^n \to \R^N$, we may want first of all to control its magnitude
by placing useful upper and lower bounds on $A$.
Such bounds are conveniently provided by the
smallest and largest singular values of $A$ denoted $\smin(A)$ and $\smax(A)$;
recall that the singular values are by definition the eigenvalues of $|A| = \sqrt{A^*A}$.

The geometric meaning of the extreme singular values can be clear
by considering the best possible factors $m$ and $M$ in the two-sided inequality
$$
m \|x\|_2 \le \|Ax\|_2 \le M \|x\|_2
\quad \text{for all } x \in \R^n.
$$
The largest $m$ and the smallest $M$ are precisely the
extreme singular values $\smin(A)$ and $\smax(A)$ respectively.
They control the distortion of the Euclidean geometry under the action of
the linear transformation $A$; the distance between any two points in $\R^n$
can increase by at most the factor $\smax(A)$ and decrease
by at most the factor $\smax(A)$.
The extreme singular values are clearly related to the operator norms of the
linear operators $A$ and $A^{-1}$ acting between Euclidean spaces:
$\smax(A) = \|A\|$  and if $A$ is invertible then $\smin(A) = 1/\|A^{-1}\|$.

Understanding the behavior of extreme singular values of random matrices is needed in many applications.
In numerical linear algebra, the {\em condition number} $\kappa(A) = \smax(A)/\smin(A)$ often
serves as a measure of stability of matrix algorithms.
Geometric functional analysis employs probabilistic constructions of linear operators as random matrices,
and the success of these constructions often depends on good bounds on the norms of these operators and
their inverses. Applications of different nature arise in statistics from the analysis of {\em sample covariance matrices} $A^*A$,
where the rows of $A$ are formed by $N$ independent samples of some unknown distribution in $\R^n$.
Some other applications are discussed in Section~\ref{s: applications}.

\paragraph{Asymptotic behavior of extreme singular values}
We first turn to the asymptotic theory for the extreme singular
values of random matrices with independent entries (and with zero mean
and unit variance for normalization purposes). From
Marchenko-Pastur law we know that most singular values of such
random $N \times n$ matrix $A$ lie in the interval
$[\sqrt{N}-\sqrt{n}, \sqrt{N} + \sqrt{n}]$. Under mild additional
assumptions, it is actually true that {\em all} singular values lie
there, so that asymptotically we have
\begin{equation}                            \label{asymptotics}
\smin(A) \sim \sqrt{N} - \sqrt{n}, \quad
\smax(A) \sim \sqrt{N} + \sqrt{n}.
\end{equation}
This fact is universal and it holds for general distributions.
This was established for $\smax(A)$ by Geman \cite{Geman}
and Yin, Bai and Krishnaiah \cite{YBK}. For $\smin(A)$,
Silverstein \cite{Silverstein} proved this for Gaussian random matrices,
and Bai and Yin \cite{Bai-Yin}
gave a unified treatment of both extreme singular values for general distributions:

\begin{theorem}[Convergence of extreme singular values, see \cite{Bai-Yin}]         \label{BY}
Let $A = A_{N,n}$ be an $N \times n$ random matrix whose entries
are independent copies of some random variable with zero mean, unit variance,
and finite fourth moment. Suppose that the dimensions $N$ and $n$ grow to infinity
while the aspect ratio $n/N$ converges to some number $y \in (0,1]$.
Then
$$
\frac{1}{\sqrt{N}} \, \smin(A) \to 1-\sqrt{y}, \quad
\frac{1}{\sqrt{N}} \, \smax(A) \to 1+\sqrt{y} \quad
\text{almost surely}.
$$
Moreover, without the fourth moment assumption the sequence
$\frac{1}{\sqrt{N}} \, \smax( A)$ is almost surely unbounded \cite{BSY}.
\end{theorem}

The limiting distribution of the extreme singular values is known and universal.
It is given by the {\em Tracy-Widom law} whose
cumulative distribution function is
\begin{equation}                        \label{F1}
F_1(x) = \exp \Big( - \int_x^\infty \big[ u(s) + (s-x) u^2(s) \big] \; ds \Big),
\end{equation}
where $u(s)$ is the solution to the Painlev\`e II equation $u'' =
2u^3 + su$ with the asymptotic $u(s) \sim \frac{1}{2 \sqrt{\pi}
s^{1/4}} \exp(-\frac{2}{3} s^{3/2})$ as $s \to \infty$. The
occurrence of Tracy-Widom law in random matrix theory and several
other areas was the subject of an ICM 2002 talk of Tracy and Widom
\cite{Tracy-Widom}. This law was initially discovered for the
largest eigenvalue of a Gaussian symmetric matrix
\cite{Tracy-Widom1, Tracy-Widom2}. For the largest singular values
of random matrices with independent entries it was established by
Johansson \cite{Johansson} and Johnstone \cite{Johnstone} in the
Gaussian case, and by Soshnihikov \cite{Soshnikov} for more general
distributions. For the smallest singular value, the corresponding
result was recently obtained in a recent work Feldheim and Sodin
\cite{Feldheim-Sodin} who gave a unified treatment of both extreme
singular values. These results are known under a somewhat stronger
subgaussian moment assumption on the entries $a_{ij}$ of $A$, which
requires their distribution to decay as fast as the normal random
variable:

\begin{definition}[Subgaussian random variables]
A random variable $X$ is {\em subgaussian} if
there exists $K>0$ called the subgaussian moment of $X$ such that
$$
\P (|X| > t) \le 2 e^{-t^2/K^2} \quad \text{for } t > 0.
$$
\end{definition}

Examples of subgaussian random variables include
normal random variables, $\pm 1$-valued, and generally,
all bounded random variables.
The subgaussian assumption is equivalent to the moment growth condition
$(\E|X|^p)^{1/p} = O(\sqrt{p})$ as $p \to \infty$.

\begin{theorem}[Limiting distribution of extreme singular values, see \cite{Feldheim-Sodin}]        \label{FS}
Let $A = A_{N,n}$ be an $N \times n$ random matrix whose entries
are  independent and identically distributed subgaussian random variables
with zero mean and unit variance.
Suppose that the dimensions $N$ and $n$ grow to infinity
while the aspect ratio $n/N$ stays uniformly bounded by some number $y \in (0,1)$.
Then the normalized extreme singular values
$$
\frac{\smin(A)^2 - (\sqrt{N} - \sqrt{n})^2}{(\sqrt{N} - \sqrt{n})(1/\sqrt{n} - 1/\sqrt{N})^{1/3}}
\quad \text{and} \quad
\frac{\smax(A)^2 - (\sqrt{N} + \sqrt{n})^2}{(\sqrt{N} + \sqrt{n})(1/\sqrt{n} + 1/\sqrt{N})^{1/3}}
$$
converge in distribution to the Tracy-Widom law \eqref{F1}.
\end{theorem}

\paragraph{Non-asymptotic behavior of extreme singular values}
It is not entirely clear to what extent the limiting behavior
of the extreme singular values such as asymptotics \eqref{asymptotics}
manifests itself in fixed dimensions. Given the geometric meaning
of the extreme singular values, our interest generally lies in establishing
correct upper bounds on $\smax(A)$ and lower bounds on $\smin(A)$.
We start with a folklore observation which yields the correct bound
$\smax(A) \lesssim \sqrt{N} + \sqrt{n}$ up to an absolute constant factor. The proof is a basic instance of
an {\em $\e$-net argument}, a technique proved to be very useful in geometric functional
analysis.

\begin{proposition}[Largest singular value of subgaussian matrices: rough bound]    \label{smax rough}
Let $A$ be an $N \times n$ random matrix whose entries
are independent mean zero subgaussian random variables
whose subgaussian moments are bounded by $1$.
Then
$$
\P \big(\smax(A) > C (\sqrt{N} + \sqrt{n}) + t \big) \le 2 e^{-c t^2},
\quad t \ge 0.
$$
Here and elsewhere in this paper, $C, C_1, c, c_1$ denote positive absolute constants.
\end{proposition}

\begin{proof}[Proof (sketch)]
We will sketch the proof for $N=n$; the general case is similar. The
expression $\smax(A) = \max_{x,y \in S^{n-1}} \< Ax, y \> $
motivates us to first control the random variables $\< Ax, y \> $
individually for each pair of vectors $x, y$ on the unit Euclidean
sphere $S^{n-1}$, and afterwards take the union bound over all such
pairs. For fixed $x,y \in S^{n-1}$ the expression $\< Ax, y \> =
\sum_{i,j} a_{ij} x_j y_i$ is a sum of independent random variables,
where $a_{ij}$ denote the independent entries of $A$. If $a_{ij}$
were standard normal random variables, the rotation invariance of
the Gaussian distribution would imply that $\< Ax, y \> $ is again a
standard normal random variable. This property generalizes to
subgaussian random variables. Indeed, using moment generating
functions one can show that a normalized sum of mean zero
subgaussian random variables is again a subgaussian random variable,
although the subgaussian moment may increase by an absolute constant
factor. Thus
$$
\P \big( \< Ax, y \> > s \big) \le 2 e^{-cs^2},
\quad s \ge 0.
$$

Obviously, we cannot finish the argument by taking the union bound
over infinite (even uncountable) number of pairs $x,y$ on the sphere
$S^{n-1}$. In order to reduce the number of such pairs, we
discretize $S^{n-1}$ by considering its $\e$-net $\NN_\e$ in the
Euclidean norm, which is a subset of the sphere that approximates
every point of the sphere up to error $\e$. An approximation
argument yields
$$
\smax(A) = \max_{x,y \in S^{n-1}} \< Ax, y \>
\le (1-\e)^{-2} \max_{x,y \in \NN_\e} \< Ax, y \>
\quad \text{for } \e \in (0,1).
$$
To gain a control over the size of the net $\NN_\e$, we construct it
as a maximal $\e$-separated subset of $S^{n-1}$; then the balls with
centers in $\NN_\e$ and radii $\e/2$ form a packing inside the
centered ball of radius $1+\e/2$. A volume comparison gives the
useful bound on the cardinality of the net: $|\NN_\e| \le
(1+2/\e)^n$. Choosing for example $\e = 1/2$, we are well prepared
to take the union bound:
\begin{align*}
\P \big( \smax(A) > 4s \big)
\le \P \big( \max_{x,y \in \NN_\e} \< Ax, y \> > s \big)
\le |\NN_\e| \max_{x,y \in \NN_\e} \P \big( \< Ax, y \> > s \big)
  \le 5^n \cdot 2 e^{-cs^2}.
\end{align*}
We complete the proof by choosing $s = C\sqrt{n} + t$ with appropriate constant $C$.
\end{proof}

By integration, one can easily deduce from Proposition~\ref{smax rough}
the correct expectation bound
$\E \smax(A) \le C_1 (\sqrt{N} + \sqrt{n})$. This latter bound
actually holds under much weaker moment assumptions.
Similarly to Theorem~\ref{BY}, the weakest possible fourth moment assumption
suffices here.
R.~Latala \cite{Latala} obtained the following general result for matrices with
not identically distributed entries:

\begin{theorem}[Largest singular value: fourth moment, non-iid entries \cite{Latala}]                 \label{Latala}
Let $A$ be a random matrix whose entries $a_{ij}$ are independent mean zero
random variables with finite fourth moment. Then
$$
\E \smax(A) \le C \Big[ \max_i \big( \sum_j \E a_{ij}^2 \big)^{1/2}
  +  \max_j \big( \sum_i \E a_{ij}^2 \big)^{1/2}
  +  \big( \sum_{i,j} \E a_{ij}^4 \big)^{1/4}  \Big].
$$
\end{theorem}

For random Gaussian matrices, a much sharper result than in Proposition~\ref{smax rough}
is due to Gordon \cite{Gordon 84, Gordon 85, Gordon 92}:

\begin{theorem}[Exteme singular values of Gaussian matrices, see \cite{DS}] \label{extreme Gaussian}
Let $A$ be an $N \times n$ matrix whose entries
are independent standard normal random variables. Then
$$
\sqrt{N} - \sqrt{n} \le \E \smin(A) \le \E \smax(A) \le \sqrt{N} + \sqrt{n}.
$$
\end{theorem}

This result is a consequence of the sharp comparison inequalities for Gaussian processes
due to Slepian and Gordon, see \cite{Gordon 84, Gordon 85, Gordon 92} and \cite [Section 3.3]{LT}.

\paragraph{Tracy-Widom fluctuations}
One can deduce from Theorem~\ref{extreme Gaussian} a deviation
inequality for the extreme singular values. It follows formally by
using the concentration of measure in the Gauss space. Since the
$\smin(A)$, $\smax(A)$ are $1$-Lipschitz functions of $A$ considered
as a vector in $\R^{Nn}$, we have
\begin{equation}        \label{subgaussian concentration}
\P \big(  \sqrt{N} - \sqrt{n} - t \le \smin(A) \le \smax(A) \le
\sqrt{N} + \sqrt{n} + t \big) \ge 1 - 2 e^{-t^2/2}, \quad t \ge 0,
\end{equation}
see \cite{DS}. For general random matrices with independent bounded
entries, one can use Talagrand's concentration inequality for convex
Lipschitz functions on the cube \cite{Tal1, Tal2}. Namely, suppose
the entries of $A$ are independent, have mean zero, and are
uniformly bounded by $1$. Since $\smax(A)$ is a convex function of
$A$, Talagrand's concentration inequality implies
\[
\P \big( |\smax(A) - \Median(\smax(A)) | \ge t \big)
\le 2 e^{-t^2/2}.
\]
Although the precise value of the median is unknown,
integration of the previous inequality shows
that $|\E \smax(A)-\Median(\smax(A))| \le C$.
The same deviation inequality holds for symmetric random matrices.

Inequality \eqref{subgaussian concentration} is optimal for large
$t$ because $\smax(A)$ is bounded below by the magnitude of every
entry of $A$ which has the Gaussian tail. But for small deviations,
say for $t<1$, inequality \eqref{subgaussian concentration} is
meaningless. Tracy-Widom law predicts a different tail behavior for
small deviations $t$. It must follow the  tail decay of the
Tracy--Widom function
$F_1$, which is not
subgaussian \cite{Aubrun}, \cite{Johnstone}:
\[
c \exp(-C \tau^{3/2}) \le 1-F_1(\tau) \le C \exp(-C' \tau^{3/2})
 \quad \tau \ge 0.
\]
The concentration of this type  for
Hermitian complex and real Gaussian matrices (Gaussian Unitary
Ensemble and Gaussian Orthogonal Ensemble) was proved by Ledoux
\cite{Ledoux-Tracy-Widom} and Aubrun \cite{Aubrun}. Recently,
Feldheim and Sodin \cite{Feldheim-Sodin} introduced a general
approach, which allows to prove  the asymptotic Tracy--Widom law and
its non-asymptotic counterpart at the same time. Moreover, their
method is applicable to random matrices with independent subgaussian
entries both in symmetric and non-symmetric case. In particular, for
an $N \times n$ random matrix $A$ with independent subgaussian
entries they proved that
\begin{equation}              \label{Frldheim-Sodin-smax}
  p(\tau):= \P \big( \smax(A) \ge \sqrt{N}+\sqrt{n}+\tau \sqrt{N} \big)
  \le C \exp ( -cn \tau^{3/2} ) \quad \tau \ge 0.
\end{equation}
Bounds \eqref{subgaussian concentration} and
\eqref{Frldheim-Sodin-smax} show that the tail behavior of the
maximal singular value is essentially different for small and large
deviations: $p(\tau)$ decays like $\exp( -cn \tau^{3/2})$ for $\tau \le
c(n/N)^2$ and like $\exp( - c_1 N \tau^2)$ for larger $\tau$.
For square matrices the meaning of this phenomenon is
especially clear.
Large deviations of $\smax(A)$ are produced
by bursts of single entries: both $\P(\smax(A) \ge \E \smax(A)+t)$
and $\P(|a_{1,1}| \ge \E \smax(A)+t)$ are of the same order $\exp(-c t^2)$
for $t \ge \E \smax(A)$. In contrast, for small deviations (for smaller $t$)
the situation becomes truly multidimensional, and
Tracy-Widom type asymptotics appears.

The method of \cite{Feldheim-Sodin} also addresses the more difficult
smallest singular value. For an $N \times n$ random matrix $A$ whose dimensions are not
too close to each other Feldheim and Sodin \cite{Feldheim-Sodin}
proved the Tracy--Widom law for the smallest singular value together
with a non-asymptotic version of the bound $\smin(A) \sim \sqrt{N} - \sqrt{n}$:
\begin{equation}    \label{Feldheim-Sodin-smin}
  \P \Big( \smin(A) \le \sqrt{N}-\sqrt{n}- \tau \sqrt{N} \cdot
  \frac{N}{N-n} \Big)
  \le \frac{C}{1-\sqrt{n/N}} \exp( -c' n \tau^{3/2} ).
\end{equation}

\section{The smallest singular value}               \label{s: smallest}

\paragraph{Qualitative invertibility problem}
In this section we focus on the behavior of the smallest singular
value of random $N \times n$ matrices with independent entries. The
smallest singular value -- the {\em hard edge} of the spectrum -- is
generally more difficult and less amenable to analysis by classical
methods of random matrix theory than the largest singular value, the
``soft edge''. The difficulty especially manifests itself for square
matrices ($N=n$) or almost square matrices ($N - n = o(n)$). For
example, we were guided so far by the asymptotic prediction
$\smin(A) \sim \sqrt{N} - \sqrt{n}$, which obviously becomes useless
for square matrices.

A remarkable example is provided by $n \times n$ {\em random
Bernoulli matrices} $A$, whose entries are independent $\pm 1$
valued symmetric random variables. Even the {\em qualitative
invertibility problem}, which asks to estimate the probability that
$A$ is invertible,  is nontrivial in this situation. Koml\'os
\cite{Komlos 67, Komlos 68} showed that $A$ is invertible
asymptotically almost surely: $p_n := \P( \smin(A) = 0) \to 0$ as $n
\to \infty$. Later Kahn, Komlos and Szemeredi \cite{KKS} proved that
the singularity probability satisfies $p_n \le c^n$ for some $c \in
(0,1)$. The base $c$ was gradually improved in \cite{TV 06, TV 07},
with the  latest record of $p_n = (1/\sqrt{2} + o(1))^n$ obtained in
\cite{BWV}. It is conjectured that the dominant source of
singularity of $A$ is the presence of two rows or two columns that
are equal up to a sign, which would imply the best possible bound
$p_n = (1/2 + o(1))^n$.

\paragraph{Quantitative invertibility problem}
The previous problem is only concerned with whether
the hard edge $\smin(A)$ is zero or not. This says nothing
about the {\em quantitative invertibility problem} of the typical size of $\smin(A)$.
The latter question has a long history.
Von~Neumann and his associates used random matrices as test inputs
in algorithms for numerical solution of systems of linear equations.
The accuracy of the matrix algorithms, and sometimes their running time as well, depends
on the condition number  $\kappa(A) = \smax(A)/\smin(A)$.
Based on heuristic and experimental evidence, von~Neumann and Goldstine
predicted that
\begin{equation}                \label{vN}
\smin(A) \sim n^{-1/2}, \quad \smax(A) \sim n^{1/2} \quad \text{with high probability}
\end{equation}
which together yield $\kappa(A) \sim n$, see \cite[Section 7.8]{vN-G}.
In Section~\ref{s: extreme} we saw several results establishing
the second part of \eqref{vN}, for the largest singular value.

Estimating the smallest singular value turned out to be more difficult.
A more precise form of the prediction $\smin(A) \sim n^{-1/2}$
was repeated by Smale \cite{Smale} and proved by Edelman \cite{Edelman}
and Szarek \cite{Sz} for {\em random Gaussian matrices}
$A$, those with i.i.d. standard normal entries. For such matrices,
the explicit formula for the joint density of the eigenvalues $\lambda_i$ of
$\frac{1}{n}A^*A$ is available:
$$
\pdf(\lambda_1,\ldots,\lambda_n) = C_n \prod_{1 \le i < j \le n}
|\lambda_i - \lambda_j|
\prod_{i=1}^n \lambda_i^{-1/2} \exp \big( - \sum_{i=1}^n \lambda_i/2 \big).
$$
Integrating out all the eigenvalues except the smallest one,
one can in principle compute its distribution.
This approach leads to the following asymptotic result:

\begin{theorem}[Smallest singular value of Gaussian matrices \cite{Edelman}]  \label{Edelman}
Let $A=A_n$ be an $n \times n$ random matrix whose entries
are independent standard normal random variables. Then for every fixed $\e \ge 0$
one has
$$
\P \big( \smin(A) \le \e n^{-1/2} \big) \to 1 - \exp(- \e - \e^2/2)
\quad \text{as } n \to \infty.
$$
\end{theorem}

The limiting probability behaves as $1 - \exp(- \e - \e^2/2) \sim
\e$ for small $\e$. In fact, the following
non-asymptotic bound holds for all $n$:
\begin{equation}                                        \label{tail}
\P \big( \smin(A) \le \e n^{-1/2} \big) \le \e,
\quad \e \ge 0.
\end{equation}
This follows from the analysis of Edelman \cite{Edelman};
Sankar, Spielman and Teng \cite{SST} provided a different geometric proof of
estimate \eqref{tail}
up to an absolute constant factor and extended it
to non-centered Gaussian distributions.

\paragraph{Smallest singular values of general random matrices}
These methods do not work for general random matrices, especially
those with discrete distributions, where rotation invariance
and the joint density of eigenvalues are not available. The prediction that
$\smin(A) \sim n^{-1/2}$ has been open even for random Bernoulli matrices.
Spielman and Teng conjectured in their ICM 2002 talk \cite{ST} that
estimate \eqref{tail} should hold for the random Bernoulli matrices up to an exponentially
small term that accounts for their singularity probability:
$$
\P \big( \smin(A) \le \e n^{-1/2} \big) \le \e + c^n, \quad \e \ge 0
$$
where $c \in (0,1)$ is an absolute constant.
The first polynomial bound on $\smin(A)$ for
general random matrices was obtained in \cite{R square}. Later
Spielman-Teng's conjecture was proved in \cite{RV square} up to a constant factor, and
for general random matrices:

\begin{theorem}[Smallest singular value of square random matrices \cite{RV square}]  \label{square}
Let $A$ be an $n \times n$ random matrix whose entries are independent
and identically distributed subgaussian random variables
with zero mean and unit variance.
Then
$$
\P \big( \smin(A) \le \e n^{-1/2} \big) \le C\e +  c^n,
\quad \e \ge 0
$$
where $C > 0$  and $c \in (0,1)$ depend only on the subgaussian
moment of the entries.
\end{theorem}

This result addresses both qualitative and quantitative aspects of the invertibility problem.
Setting $\e=0$ we see that $A$ is invertible with probability at least $1-c^n$.
This generaizes the result of Kahn, Komlos and Szemeredi \cite{KKS} from
Bernoulli to all subgaussian matrices.
On the other hand, quantitatively, Theorem~\ref{square}
states that $\smin(A) \gtrsim n^{-1/2}$ with high probability for general random matrices.
A corresponding non-asymptotic upper bound $\smin(A) \lesssim n^{-1/2}$ also holds \cite{RV CRAS},
so we have $\smin(A) \sim n^{-1/2}$ as in von~Neumann-Goldstine's prediction.
Both these bounds, upper and lower, hold with high probability under
the weaker fourth moment assumption on the entries \cite{RV square, RV CRAS}.

This theory was extended to {\em rectangular random matrices}
of arbitrary dimensions $N \times n$ in \cite{RV rectangular}. As we know from
Section~\ref{s: extreme}, one expects that  $\smin(A) \sim \sqrt{N} - \sqrt{n}$.
But this would be incorrect for square matrices.
To reconcile rectangular and square matrices we
make the following correction of our prediction:
\begin{equation}                                        \label{smin rectangular}
\smin(A) \sim \sqrt{N} - \sqrt{n-1}
\quad \text{with high probability}.
\end{equation}
For square matrices one would have the correct estimate
$\smin(A) \sim \sqrt{n} - \sqrt{n-1} \sim n^{-1/2}$. The following result
extends Theorem~\ref{square} to rectangular matrices:

\begin{theorem}[Smallest singular value of rectangular random matrices \cite{RV square}]
\label{rectangular}
Let $A$ be an $n \times n$ random matrix whose entries are independent
and identically distributed subgaussian random variables
with zero mean and unit variance.
Then
$$
\P \big( \smin(A) \le \e (\sqrt{N} - \sqrt{n-1}) \big) \le (C\e)^{N-n+1} +  c^N,
\quad \e \ge 0
$$
where $C > 0$  and $c \in (0,1)$ depend only on the subgaussian
moment of the entries.
\end{theorem}

This result has been known for a long time for {\em tall matrices},
whose the aspect ratio $\lambda = n/N$ is bounded by a sufficiently
small constant, see \cite{Bennett}. The optimal bound $\smin(A) \ge
c \sqrt{N}$ can be proved in this case using an $\e$-net argument
similar to Proposition~\ref{smax rough}. This was extended in
\cite{LPRT} to $\smin(A) \ge c_\l \sqrt{N}$ for all aspect ratios
$\lambda < 1 - c/\log n$. The dependence of $c_\l$ on the aspect
ratio $\l$ was improved in \cite{AFMS} for Bernoulli matrices and in
\cite{R 06} for general subgaussian matrices. Feldheim-Sodin's
Theorem~\ref{FS} gives precise Tracy-Widom fluctuations of
$\smin(A)$ for tall matrices, but becomes useless for almost square
matrices (say for $N < n+ n^{1/3}$). Theorem~\ref{rectangular} is an an optimal
result (up to absolute constants) which covers matrices with all
aspect ratios from tall to square. Non-asymptotic estimate
\eqref{smin rectangular} was extended to matrices whose entries have
finite $(4+\e)$-th moment in \cite{V product}.

\paragraph{Universality of the smallest singular values}
The limiting distribution of $\smin(A)$ turns out to be universal as dimension $n \to \infty$.
We already saw a similar universality phenomenon in Theorem~\ref{FS}
for genuinely rectangular matrices. For square matrices, the corresponding result
was proved by Tao and Vu \cite{TV 09}:

\begin{theorem}[Smallest singular value of square matrices: universality \cite{TV 09}]
Let $A$ be an $n \times n$ random matrix whose entries are independent
and identically distributed random variables with zero mean, unit variance,
and finite $K$-th moment where $K$ is a sufficiently large absolute constant.
Let $G$ be an $n \times n$ random matrix whose entries
are independent standard normal random variables.
Then
$$
\P (\sqrt{n} \smin(G) \le t - n^{-c}) - n^c
\le \P (\sqrt{n}\smin(A) \le t)
\le \P (\sqrt{n} \smin(G) \le t + n^{-c}) + n^c
$$
where $c>0$ depends only on the $K$-th moment of the entries.
\end{theorem}

On a methodological level, this result may be compared in classical
probability theory to Berry-Esseen theorem \eqref{Berry-Esseen}
which establishes polynomial deviations from the limiting distribution,
while Theorems~\ref{square} and \ref{rectangular} bear a similarity with
large deviation results like \eqref{deviation} which give exponentially
small tail probabilities.

\paragraph{Sparsity and invertibility: a geometric proof of Theorem~\ref{square}}
We will now sketch the proof of Theorem~\ref{square} given in \cite{RV square}.
This argument is mostly based on geometric ideas, and it may be useful beyond
spectral analysis of random matrices.

Looking at  $\smin(A) = \min_{x \in S^{n-1}} \|Ax\|_2$
we see that our goal is to bound below $\|Ax\|_2$ uniformly for all unit vectors $x$.
We will do this separately for {\em sparse vectors} and for {\em spread vectors}
with two very different arguments. Choosing a small absolute constant $c_0>0$,
we first consider the class of sparse vectors
$$
\Sparse := \{ x \in S^{n-1} :\; |\supp(x)| \le c_0 n \}
$$
Establishing invertibility of $A$ on this class is relatively easy.
Indeed, when we look at $\|Ax\|_2$ for sparse vectors $x$ of fixed support $\supp(x)=I$
of size $|I| = c_0 n$,
we are effectively dealing with the $n \times c_0 n$ submatrix $A_I$ that consists
of the columns of $A$ indexed by $I$. The matrix $A_I$ is tall, so as we said below
Theorem~\ref{rectangular}, its smallest singular value can be estimated using
the standard $\e$-net argument. This gives $\smin(A_I) \ge c n^{1/2}$ with
probability at least $1 - 2e^{-n}$. This allows us to further take the union bound over
$\binom{n}{c_0 n} \le e^{n/2}$ choices of support $I$, and conclude that
with probability at least $1 - 2e^{-n/2}$ we have invertibility on all sparse vectors:
\begin{equation}                            \label{on sparse}
\min_{x \in \Sparse} \|Ax\|_2 = \min_{|I| \le c_0 n} \smin(A_I) \ge c n^{1/2}.
\end{equation}
We thus obtained a much stronger bound than we need, $n^{1/2}$ instead of $n^{-1/2}$.

Establishing invertibility of $A$ on non-sparse vectors is more difficult
because there are too many of them. For example, there are exponentially many
vectors on $S^{n-1}$ whose coordinates all equal $\pm n^{-1/2}$
and which have at least a constant distance from each other.
This gives us no hope to control such vectors using $\e$-nets, as any
nontrivial net must have cardinality at least $2^n$. So let us now focus on this most
difficult class of extremely non-sparse vectors
$$
\Spread := \{ x \in S^{n-1} :\; |x_i| \ge c_1 n^{-1/2} \text{ for
all } i \}.
$$
Once we prove invertibility of $A$ on these spread vectors, the argument can be completed for all vectors
in $S^{n-1}$ by an approximation argument. Loosely speaking,
if $x$ is close to $\Sparse$ we can treat $x$ as sparse,
otherwise $x$ must have at least $cn$ coordinates of magnitude $|x_i| = O(n^{-1/2})$, which
allows us to treat $x$ as spread.

An obvious advantage of spread vectors is that we know the magnitude
of all their coefficients. This motivates the following geometric invertibility argument.
If $A$ performs extremely poor so that $\smin(A) = 0$, then
one of the columns $X_k$ of $A$ lies in the span $H_k = \Span(X_i)_{i \ne k}$
of the others. This simple observation can be transformed into a quantitative argument.
Suppose $x = (x_1,\ldots,x_n) \in \R^n$ is a spread vector.
Then, for every $k=1,\ldots,n$, we have
\begin{align}                   \label{Ax spread}
\|Ax\|_2
&\ge \dist(Ax, H_k)
= \dist \Big( \sum_{i=1}^n x_i X_i, H_k \Big)
= \dist (x_k X_k, H_k)  \nonumber\\
&= |x_k| \cdot \dist (X_k, H_k)
\ge c_1 n^{-1/2} \; \dist(X_k, H_k).
\end{align}
Since the right hand side does not depend on $x$, we have proved
that
\begin{equation}                            \label{on spread}
\min_{x \in \Spread} \|Ax\|_2 \ge c_1 n^{-1/2} \; \dist(X_n, H_n).
\end{equation}

This reduces our task to the geometric problem of independent interest --
estimate {\em the distance between a random vector and an independent random hyperplane}.
The expectation estimate $1 \le \E \dist(X_n, H_n)^2 = O(1)$
follows easily by independence and moment assumptions.
But we need a lower bound with high probability, which is far from trivial.
This will make a separate story connected to
the Littlewood-Offord theory of {\em small ball probabilities},
which we discuss in Section~\ref{s: Littlewood-Offord}.
In particular we will prove in Corollary~\ref{distance} the optimal estimate
\begin{equation}                                        \label{missing distance}
\P (\dist(X_n, H_n) \le \e) \le C \e + c^n,
\quad \e \ge 0,
\end{equation}
which is simple for the Gaussian distribution (by rotation invariance)
and difficult to prove e.g. for the Bernoulli distribution.
Together with \eqref{on spread} this means that we proved invertibility on all spread vectors:
$$
\P \big( \min_{x \in \Spread} \|Ax\|_2 \le \e n^{-1/2} \big) \le C\e +  c^n,
\quad \e \ge 0.
$$
This is exactly the type of probability bound claimed in Theorem~\ref{square}.
As we said, we can finish the proof by
combining with the (much better) invertibility on sparse vectors in \eqref{on sparse},
and by an approximation argument.

\section{Littlewood-Offord theory}          \label{s: Littlewood-Offord}

\paragraph{Small ball probabilities and additive structure}
We encountered the following geometric problem in the previous section:
{\em estimate the distance between a random vector $X$ with independent
coordinates and an independent random hyperplane $H$ in $\R^n$}.
We need a lower bound on this distance with high probability.
Let us condition on the hyperplane $H$ and let
$a \in \R^n$ denote its unit normal vector.
Writing in coordinates  $a = (a_1,\ldots,a_n)$ and $X = (\xi_1, \ldots, \xi_n)$,
we see that
\begin{equation}                                        \label{dist sum}
\dist(X, H) = \< a, X \>  =  \Big| \sum_{i=1}^n a_i \xi_i \big|.
\end{equation}
We need to understand the distribution of sums of independent random
variables
$$
S = \sum_{i=1}^n a_i \xi_i, \quad \|a\|_2 = 1,
$$
where $a = (a_1,\ldots,a_n) \in \R^n$ is a given coefficient vector, and $\xi_1, \ldots, \xi_n$ are
independent identically distributed random variables with zero mean and unit variance.

Sums of independent random variables is a classical theme in probability theory.
The well-developed area of large deviation
inequalities like \eqref{deviation} demonstrates
that $S$ nicely concentrates around its mean.
But our problem is opposite as we need to show that $S$ is {\em not too concentrated}
around its mean $0$, and perhaps more generally around any real number.
Several results in probability theory
starting from the works of L\'evy \cite{Lev}, Kolmogorov \cite{Kol} and Ess\'{e}en \cite{Esseen 66}
were concerned with the spread of sums of independent random variables,
which is quantified as follows:

\begin{definition}
The {\em L\'evy concentration function} of a random variable $S$ is
$$
\LL(S, \e) = \sup_{v \in \R} \P ( |S-v| \le \e ),
\quad \e \ge 0.
$$
\end{definition}

L\'evy concentration function measures the {\em small ball probability} \cite{LS}, the likelihood
that $S$ enters a small interval.
For continuous distributions one can show that $\LL(S,\e) \lesssim \e$ for all $\e \ge 0$.
For discrete distributions this may be false. Instead, a new phenomenon arises
for discrete distributions which is unseen in large deviation theory: L\'evy concentration
function depends on the {\em additive structure} of the coefficient vector $a$.
This is best illustrated on the example where
$\xi_i$ are independent Bernoulli random variables ($\pm 1$ valued and symmetric).
For sparse vectors like $a = 2^{-1/2}(1,1,0,\ldots,0)$, L\'evy concentration function
can be large: $\LL(S,0) = 1/2$.
For spread vectors, Berry-Esseen's theorem \eqref{Berry-Esseen} yields a better bound:
\begin{equation}                            \label{a'}
\text{For } a' = n^{-1/2} (1,1,\ldots,1), \quad
\LL(S,\e) \le C(\e + n^{-1/2}).
\end{equation}
The threshold $n^{-1/2}$ comes from many cancelations in the sums $\sum \pm 1$ which occur
because all coefficients $a_i$ are equal. For less structured $a$,
fewer cancelations occur:
\begin{equation}                            \label{a''}
\text{For } a'' = n^{-1/2} \big(1 + \frac{1}{n}, 1 + \frac{2}{n}, \ldots, 1 + \frac{n}{n} \big), \quad
\LL(S,0) \sim n^{-3/2}.
\end{equation}
Studying the influence of additive structure of the coefficient vector $a$
on the spread of $S = \sum a_i \xi_i$
became known as the {\em Littlewood-Offord problem}. It was initially
developed by Littlewood and Offord \cite{LO}, Erd\"os and Moser \cite{Erdos 45, Erdos 65},
S\'arkozy and Szemer\'edi \cite{SS}, Halasz \cite{Halasz}, Frankl and F\"uredi \cite{FF}.
For example, if all $|a_i| \ge 1$ then $\LL(S,1) \le Cn^{-1/2}$ \cite{LO, Erdos 45},
which agrees with \eqref{a'}.
Similarly, a general fact behind \eqref{a''} is that if $|a_i - a_j| \ge 1$ for all $i \ne j$
then $\LL(S,1) \le Cn^{-3/2}$ \cite{Erdos 65, SS, Halasz}.

\paragraph{New results on L\'evy concentration function}
Problems of invertibility of random matrices motivated a recent revisiting 
of the Littlewood-Offord problem by Tao and Vu \cite{TV Annals, TV
Bulletin, TV RSA, TV FF}, the authors \cite{RV square, RV
rectangular}, Friedland and Sodin \cite{Friedland-Sodin}. Additive
structure of the coefficient vector $a$ is related to the shortest
arithmetic progression into which it embeds. This length is
conveniently expressed as the {\em least common denominator}
$\lcd(a)$ defined as the smallest $\theta > 0$ such that $\theta a
\in \Z^n \setminus 0$. Examples suggest that L\'evy concentration
function should be inversely proportional to the least common
denominator: $\lcd(a') = n^{1/2} \sim 1/\LL(S,0)$ in \eqref{a'} and
$\lcd(a'') = n^{3/2} \sim 1/\LL(S,0)$ in \eqref{a''}. This is not a
coincidence. But to state a general result, we will need to consider
a more stable version of the least common denominator. Given an
accuracy level $\a > 0$, we define the {\em essential least common
denominator}
$$
\lcd_\a(a) := \inf \big\{ \theta > 0: \; \dist(\theta a, \Z^n) \le \min( \frac{1}{10} \|\theta a\|_2, \a) \big\}.
$$
The requirement $\dist(\theta a, \Z^n) \le \frac{1}{10} \|\theta
a\|_2$ ensures approximation of $\theta a$ by non-trivial integer
points, those in a non-trivial cone in the direction of $a$. The
constant $\frac{1}{10}$ is arbitrary and it can be replaced by any
other constant in $(0,1)$. One typically uses this concept for
accuracy levels $\a = c \sqrt{n}$ with a small constant $c$ such as
$c = \frac{1}{10}$. The inequality $\dist(\theta a, \Z^n) \le \a$
yields that most of the coordinates of $\theta a$ are within a small
constant distance from integers. For such $\a$, in examples
\eqref{a'} and \eqref{a''} one has as before $\lcd_\a(a') \sim
n^{1/2}$ and $\lcd_\a(a'') \sim n^{3/2}$. Here we state and sketch a
proof of a general Littlewood-Offord type result from \cite{RV
rectangular}.

\begin{theorem}[L\'evy concentration function via additive structure]       \label{sbp}
Let $\xi_1, \ldots, \xi_n$ be independent identically distributed mean zero random variables,
which are well spread: $p: = \LL(\xi_k, 1) < 1$.
Then, for every coefficient vector $a = (a_1,\ldots,a_n) \in S^{n-1}$
and every accuracy level $\a > 0$, the sum $S = \sum_{i=1}^n a_i \xi_i$ satisfies
\begin{equation}                                       \label{sbp eq}
\LL(S,\e) \le C \e + C/\lcd_\a(a) + C e^{- c\a^2},
\quad \e \ge 0,
\end{equation}
where $C, c > 0$ depend only on the spread $p$.
\end{theorem}

\begin{proof}
A classical Esseen's concentration inequality \cite{Esseen 66}
bounds the L\'evy concentration function of an arbitrary random
variable $Z$ by the $L_1$ norm of its characteristic function
$\phi_Z(\theta)=\E \exp ( i \theta Z)$ as follows:
\begin{equation}                                        \label{Esseen}
\LL(Z,1) \le C \int_{-1}^{1} |\phi_Z(\theta)| \, d \theta.
\end{equation}
One can prove this inequality using Fourier inversion formula, see \cite[Section 7.3]{TV book}.

We will show how to prove Theorem~\ref{sbp} for Bernoulli random
variables $\xi_i$; the general case   requires an additional
argument. Without loss of generality we can assume that
$\lcd_{\a}(a) \ge \frac{1}{\pi \e}$. Applying \eqref{Esseen} for
$Z=S/\e$, we obtain by independence that
$$
\LL(S,\e) \le C \int_{-1}^{1} |\phi_S(\theta/\e)| \, d \theta =C
\int_{-1}^{1} \prod_{j=1}^n |\phi_j(\theta/\e)| \, d \theta,
$$
where $\phi_j(t)=\E \exp( ia_j \xi_j t)=\cos(a_j t)$. The inequality
$|x| \le \exp(-\frac{1}{2}(1-x^2))$ which is valid for all $x \in
\R$ implies that
$$
|\phi_j(t)|
\le \exp \Big( - \frac{1}{2} \sin^2 (a_j t) \Big)
\le \exp \Big( -\frac{1}{2} \dist(\frac{a_j t}{\pi}, \Z)^2 \Big).
$$
Therefore
\begin{equation}                                        \label{LL}
\LL(S,\e)
\le C \int_{-1}^1 \exp \Big( -\frac{1}{2}
    \sum_{j=1}^n  \dist \big(\frac{a_j \theta}{\pi \e}, \Z \big)^2 \Big) \, d\theta
= C \int_{-1}^{1}
  \exp \Big( -\frac{1}{2} f^2(\theta) \Big) \, d \theta
\end{equation}
where $f(\theta) = \dist \big( \frac{\theta}{\pi \e} a, \Z^n \big)$.
Since $\lcd_{\a}(a) \ge \frac{1}{\pi \e}$, the definition of the essential least
common denominator implies that for every $\theta \in [-1,1]$ we have
$f(\theta) \ge \min (\frac{\theta}{10 \pi \e} \|a\|_2, \a)$.
Since by assumption $\|a\|_2=1$, it follows that
$$
\exp \Big( -\frac{1}{2} f^2(\theta) \Big)
\le \exp \Big( -\frac{1}{2} \Big( \frac{\theta}{10 \pi \e} \Big)^2 \Big)
+\exp (-\a^2/2).
$$
Substituting this into \eqref{LL} yields $\LL(S,\e) \le C_1 (\e + 2\exp (-\a^2/2))$
as required.
\end{proof}

Theorem ~\ref{sbp} justifies our empirical observation that L\'evy concentration function
is proportional to the amount of structure in the coefficient vector,
which is measured by the (reciprocal of) its essential least common denominator.
As we said, this result is typically used for accuracy level
$\a = c \sqrt{n}$ with some small positive constant $c$. In this case, the term
$C e^{- c\a^2}$ in \eqref{sbp eq} is exponentially small in $n$ (thus negligible in applications),
and the term $C\e$ is optimal for continuous distributions.

Theorem~\ref{sbp} performs best for totally unstructured coefficient vectors $a$,
those with exponentially large $\lcd_\a(a)$. Heuristically, this should be the case
for random vectors,  as randomness should destroy any structure. While this is not
true for general vectors with independent coordinates (e.g. for equal
coordinates with random signs), it is true for {\em normals of random hyperplanes}:

\begin{theorem}[Random vectors are unstructured \cite{RV square}]                \label{unstructured}
Let $X_i$ be random vectors in $\R^n$ whose coordinates
are independent and identically distributed subgaussian random variables
with zero mean and unit variance. Let $a \in \R^n$ denote a unit normal vector
of $H = \Span(X_1,\ldots,X_{n-1})$. Then, with probability at least $1 - e^{-cn}$,
$$
\lcd_\a (a) \ge e^{cn} \quad \text{for } \a = c \sqrt{n},
$$
where $c>0$ depends only on the subgaussian moment.
\end{theorem}

Therefore for random normals $a$, Theorem~\ref{sbp} yealds with high
probability
a very strong bound on L\'evy concentration function:
\begin{equation}                                        \label{best LL}
\LL(S,\e) \le C \e + c^n,
\quad \e \ge 0.
\end{equation}
This brings us back to the distance problem considered in the beginning of this
section, which motivated our study of L\'evy concentration function:

\begin{corollary}[Distance between random vectors and hyperplanes \cite{RV square}]   \label{distance}
Let $X_i$ be random vectors as in Theorem~\ref{unstructured}, and
$H_n = \Span(X_1, \ldots, X_{n-1})$. Then
$$
\P \big( \dist(X_n, H_n) \le \e \big) \le C \e + c^n,
\quad \e \ge 0,
$$
where $C, c>0$ depend only on the subgaussian moment.
\end{corollary}

\begin{proof}
As was noticed in \eqref{dist sum}, we can write $\dist(X_n, H_n)$
as a sum of independent random variables, and then bound it using
\eqref{best LL}.
\end{proof}

Corollary~\ref{distance} offers us exactly the missing piece \eqref{missing distance}
in our proof of the invertibility Theorem~\ref{square}. This completes our analysis
of invertibility of square matrices.

\begin{remark}
These methods generalize to rectangular matrices \cite{RV rectangular, V product}.
For example, Corollary~\ref{distance} can be extended to compute the distance between
random vectors and subspaces of arbitrary dimension \cite{RV rectangular}:
for $H_n = \Span(X_1, \ldots, X_{n-d})$ we have $(\E \dist(X_n, H_n)^2)^{1/2} = \sqrt{d}$
and
$$
\P \big( \dist(X_n, H_n) \le \e \sqrt{d} \big) \le (C \e)^d + c^n,
\quad \e \ge 0.
$$
\end{remark}

\section{Applications}              \label{s: applications}

The applications of non-asymptotic theory of random matrices are numerous, and we cannot cover
all of them in this note. Instead we concentrate on three different results pertaining to
the classical random matrix theory (Circular Law), signal processing (compressed sensing), and
geometric functional analysis and theoretical computer science (short Khinchin's inequality and 
Kashin's subspaces).

\paragraph{Circular law}
Asymptotic theory of random matrices provides an important source of heuristics for
non-asymptotic results. We have seen an illustration of this in the analysis of the extreme singular
values. This interaction between the asymptotic and non-asymptotic
theories goes the other way as well, as good non-asymptotic bounds are sometimes crucial in proving
the limit laws. One remarkable example of this is the circular law which we will discuss now.

Consider a family of $n \times n$ matrices $A$ whose entries are independent copies of a
random variable $X$ with mean zero and unit variance. Let $\mu_n$ be the empirical measure of the eigenvalues
of the matrix $B_n = \frac{1}{\sqrt{n}}A_n$, i.e. the Borel probability measure on $\C$ such that $\mu_n(E)$
is the fraction of the eigenvalues of $B_n$ contained in $E$.
A long-standing conjecture in random matrix theory, which is called the circular law,
suggested that {\em the measures $\mu_n$ converge to the normalized Lebesgue measure on the unit
disc}. The convergence here can be understood in the same sense as in the Wigner's semicircle law.
The circular law was originally proved by Mehta \cite{Mehta} for random matrices with standard normal
entries. The argument used the explicit formula for joint density of the eigenvalues, so it
could not be extended to other classes of random matrices. While the formulation of Wigner's semicircle law
and the circular law look similar, the methods used to prove the former are not applicable to the latter.
The reason is that the spectrum of a general matrix, unlike that of a Hermitian matrix, is unstable:
a small change of the entries may cause a significant change of the spectrum (see \cite{BS}).
Girko \cite{Girko}
introduced a new approach to the circular law based on considering the real part of
the Stieltjes transform of measures $\mu_n$. For $z=x+iy$ the real Stieltjes transform is defined
by the formula
\[
S_{nr}(z)=\text{Re}  \Big( \frac{1}{n} \text{Tr} (B_n- z I_n)^{-1} \Big)
=  -\frac{\partial}{\partial x} \Big( \frac{1}{n} \log |\text{det}(B_n-zI)| \Big).
\]
Since $|\text{det}(B_n-zI)|^2=\text{det} (B_n-zI)(B_n-zI)^* $, this is the same as
\[
S_{nr}(z)
  = -\frac{1}{2} \frac{\partial}{\partial x} \Big( \frac{1}{n} \log |\text{det}(B_n-zI)(B_n-zI)^*| \Big)
  = - \frac{1}{2} \frac{\partial}{\partial x} \Big( \frac{1}{n} \sum_{j=1}^n \log s_j^{(n)}(z) \Big),
\]
where $s_1^{(n)}(z) \ge \ldots \ge s_n^{(n)}(z) \ge 0$
are the eigenvalues of the {\em Hermitian} matrix $(B_n-zI)(B_n-zI)^*$,
or in other words, the squares of the singular values of the matrix $V_n=B_n-zI$.
Girko's argument reduces the proof of the circular law to the convergence of real Stieltjes transforms,
and thus to the behavior of the sum above.
The logarithmic function is unbounded at $0$ and $\infty$.
To control the behavior near $\infty$, one has to use the bound
for the largest singular value of $V_n$, which is relatively easy.
The analysis of the behavior near $0$ requires bounds on the
smallest singular value of $V_n$, and is therefore more difficult.

Girko's approach was implemented by Bai \cite{Bai 97}, who proved
the circular law for random matrices whose entries have bounded sixth moment
and bounded density. The bounded density condition was sufficient to take care
of the smallest singular value problem.
This result was the first manifestation of the universality of the circular law.
Still, it did not cover some important classes of random matrices,
in particular random Bernoulli matrices. 
The recent results on the smallest singular value led to a significant progress on establishing the universality of the circular law.
A crucial step was done by G\"otze and Tikhomirov \cite {Gotze-Tikhomirov 07}
who extended the circular law to all subgaussian matrices using \cite{R square}.
In fact, the results of \cite{Gotze-Tikhomirov 07} actually extended it to all random
entries with bounded fourth moment. This was further extended to random variables having
bounded moment $2+\e$ in \cite{Gotze-Tikhomirov 08, TV Circular1}.
Finally, in \cite{TV Circular2} Tao and Vu  proved the Circular Law in full generality,
with no assumptions besides the unit variance. Their approach was based on the smallest singular value bound from \cite{TV Circular1} and a novel {\em replacement principle} which allowed them to treat the other singular values.

\paragraph{Compressed Sensing}
Non-asymptotic random matrix theory provides a right context
for the analysis of random measurements in the newly developed area of compressed sensing,
see the ICM 2006 talk of Candes \cite{Candes ICM}.
Compressed sensing is an area of information theory and signal processing
which studies efficient techniques to reconstruct a signal
from a small number of measurements by utilizing the prior knowledge that
the signal is sparse \cite{Candes-Wakin intro}.

Mathematically, one seeks to reconstruct an unknown signal $x \in \R^n$
from some $m$ linear measurements viewed as a vector $Ax \in \R^m$,
where $A$ is some known $m \times n$ matrix called the {\em measurement matrix}.
In the interesting case $m < n$, the
problem is underdetermined and we are interested in the sparsest solution:
\begin{equation}                    \label{non-convex}
\text{minimize }  \|x^*\|_0 \text{ subject to } Ax^* = Ax,
\end{equation}
where $\|x\|_0 = |\supp(x)|$.
This optimization problem is highly non-convex and computationally intractable.
So one considers the following convex relaxation of \eqref{non-convex}, which can be efficiently solved
by convex programming methods:
\begin{equation}                    \label{convex}
\text{minimize }  \|x^*\|_1 \text{ subject to } Ax^* = Ax,
\end{equation}
where $\|x\|_1 = \sum_{i=1}^n |x_i|$ denotes the $\ell_1$ norm.

One would then need to find conditions when
problems \eqref{non-convex} and \eqref{convex} are equivalent.
Candes and Tao \cite{CT decoding} showed that this occurs when the measurement matrix $A$ is
a {\em restricted isometry}. For an integer $s \le n$,
the restricted isometry constant $\d_s(A)$ is the smallest number $\d \ge 0$
which satisfies
\begin{equation}                            \label{RIP}
(1-\d) \|x\|_2^2 \le \|Ax\|_2^2 \le (1+\d) \|x\|_2^2
\quad \text{for all } x \in \R^n, \; |\supp(x)| \le s.
\end{equation}
Geometrically, the restricted isometry property guarantees that the geometry of
$s$-sparse vectors $x$ is well preserved by the measurement matrix $A$.
In turns out that in this situation one can reconstruct $x$
from $Ax$ by the convex program \eqref{convex}:

\begin{theorem}[Sparse reconstruction using convex programming \cite{CT decoding}]
Assume $\d_{2s} \le c$. Then the solution of \eqref{convex}
equals $x$ whenever $|\supp(x)| \le s$.
\end{theorem}

A proof with $c = \sqrt{2}-1$ is given in \cite{Candes sqrt2}; the current record is $c = 0.472$ \cite{CWX}.

Restricted isometry property can be interpreted in terms of the
{\em extreme singular values of submatrices} of $A$.
Indeed, \eqref{RIP} equivalently states that the inequality
$$
\sqrt{1-\d} \le \smin(A_I) \le \smax(A_I) \le \sqrt{1+\d}
$$
holds for all $m \times s$ submatrices $A_I$,
those formed by the columns of $A$ indexed by sets $I$ of size $s$.
In light of Sections~\ref{s: extreme} and \ref{s: smallest},
it is not surprising that the best known restricted isometry matrices
are {\em random matrices}. It is actually an open problem to construct {\em deterministic}
restricted isometry matrices as in Theorem~\ref{RIP matrices} below.

The following three types of random matrices
are extensively used as measurement matrices in compressed sensing:
{\em Gaussian, Bernoulli, and Fourier}. Here we summarize their restricted
isometry properties, which have the common remarkable feature:
{\em the required number of measurements $m$ is roughly proportional to the sparsity level $s$}
rather than the (possibly much larger) dimension $n$.

\begin{theorem}[Random matrices are restricted isometries]          \label{RIP matrices}
Let $m,n,s$ be positive integers, $\e, \d \in (0,1)$,
and let $A$ be an $m \times n$ measurement matrix.

1. Suppose the entries of $A$ are independent and identically distributed
subgaussian random variables with zero mean and unit variance.
Assume that
$$
m \ge C s \log (2n/s)
$$
where $C$ depends only on $\e$, $\d$, and the subgaussian moment.
Then with probability at least $1-\e$,
the matrix $\bar{A} = \frac{1}{\sqrt{m}} A$ is a restricted isometry with $\d_s(\bar{A}) \le \d$.

2. Let $A$ be a random Fourier matrix obtained from the
$n \times n$ discrete Fourier transform matrix by choosing $m$ rows independently
and uniformly. Assume that
\begin{equation}                          \label{log4}
m \ge C s \log^4 (2n).
\end{equation}
where $C$ depends only on $\e$ and $\d$. Then with probability at least $1-\e$,
the matrix $\bar{A} = \frac{1}{\sqrt{n}} A$ is a restricted isometry with $\d_s(\bar{A}) \le \d$.
\end{theorem}

For random subgaussian matrices this result was proved in \cite{BDDW, MPT} by an $\e$-net argument,
where one first checks the deviation inequality $|\|Ax\|_2^2 - 1| \le \d$ with exponentially
high probability for a fixed vector $x$ as in \eqref{RIP}, and afterwards lets $x$ run over
some fine net. For random Fourier matrices the problem is harder. It was first
addressed in \cite{CT Fourier} with a little higher exponent than in \eqref{log4};
the exponent $4$ was obtained in \cite{RV Fourier}, and it is conjectured that the optimal exponent is $1$.

\paragraph{Short Khinchin's inequality and Kashin's subspaces}
Let $1 \le p < \infty$. The classical Khinchin's inequality states that there exist constants $A_p,B_p$
such that for all $x=(x_1, \ldots, x_n) \in \R^n$
\[
 A_p \|x\|_2
 \le  \Big( \Ave_{\e \in \{-1,1\}^n}
  \Big| \sum_{j=1}^n \e_j x_j  \Big|^p \Big)^{1/p}
 \le B_p \|x\|_2.
\]
The average here is taken over all $2^n$ possible choices of signs $\e$ (it is the same as
the expectation with respect to independent Bernoulli random variables $\e_j$).
Since the mid-seventies, the question was around whether Khinchin's inequality
holds for averages over some small sets of signs $\e$.
A trivial lower bound follows by a dimension argument: such a set must contain at least $n$ points.
Here we shall discuss only the case $p=1$, which is of considerable interest for computer science.
This problem can be stated more precisely as follows:
as follows:
\begin{quote}
 Given $\d>0$, find $\a(\d), \b(\d)>0$ and construct a set $V \subset \{-1,1\}^n$ of cardinality
 less than $(1+\d)n$ such that for all $x=(x_1, \ldots, x_n) \in \R^n$
 \begin{equation}     \label{short Khinchin}
   \a(\d) \|x\|_2
   \le  \Ave_{\e \in V}
   \Big| \sum_{j=1}^n \e_j x_j  \Big|
   \le \b(\d) \|x\|_2.
 \end{equation}
\end{quote}
The first result in this direction belongs to Schechtman
\cite{Schechtman} who found an affirmative solution to this problem
for $\d$ greater than some absolute constant. He considered a set
$V$ consisting of $N=\lfloor (1+\d)n \rfloor$ independent random
$\pm 1$ vectors, which can be written as an $N \times n$ random
Bernoulli matrix $A$. In the matrix language, the inequality above
reads $\a(\d) \|x\|_2  \le N^{-1} \| Ax \|_1 \le \b(\d)\|x\|_2 $ for
all $x \in \R^n$. This means that one can take
\[
 \a(\d)= N^{-1} \inf_{x \in S^{n-1}} \|Ax\|_1,
 \quad
 \b(\d)= N^{-1} \sup_{x \in S^{n-1}} \|Ax\|_1.
\]
These expressions bear a similarity to the smallest and the largest singular values of the matrix $A$.
In fact, up to the coefficient $N^{-1}$, $\b(\d)$ is the norm of $A$ considered as a linear operator
from $\ell_2^n$ to $\ell_1^n$, and $\a(\d)$ is the reciprocal of the norm of its inverse. Schechtman's
theorem can now be derived using the $\e$-net argument.

The case of small $\d$ is more delicate. For a random $A$, the bound
for $\b(\d)\le C$ can be obtained by the $\e$-net argument as
before. However, an attempt to apply this argument for $\a(\d)$ runs
into to the same problems as for the smallest singular value. For
any fixed  $\d>0$ the solution was first obtained first by Johnson and
Schechtman \cite{JS} who showed that there {\em exists} $V$
satisfying \eqref{short Khinchin} with $\a(\d)=c^{1/\d}$. In
\cite{LPRTV} this was established for a random set $V$ (or a random
matrix $A$) with the same bound on $\a(\d)$.
Furthermore, the result remains valid even when $\d$ depends on
$n$, as long as $\d \ge c/\log n$.
The proof uses the
smallest singular value bound from \cite{LPRT} in a crucial way. The
bound on $\a(\d)$ has been further improved in \cite{AFMS}, also
using the singular value approach. Finally, a theorem in \cite{R 06}
asserts that for a random set $V$ the inequalities \eqref{short
Khinchin} hold with high probability for 
$$
\a(\d)=c\d^2, \quad 
\b(\d)=C.
$$
Moreover, the result holds for all  $\d>0$ and $n$, without any
restrictions.
The proof
combines the methods of \cite{R square} and a geometric argument
based on the structure of a section of the $\ell_1^n$ ball. The
probability estimate of \cite{R 06} can be further improved if one
replaces the small ball probability bound of \cite{R square} with
that of \cite{RV square}.

The short Khinchin inequality shows also that the $\ell_1$ and $\ell_2$ norms are equivalent on a
random subspace $E:=A \R^n \subset \R^N$. Indeed, if $A$ is an $N \times n$ random matrix, then with
high probability every vector $x \in \R^n$ satisfies
$\a(\d) \|x\|_2 \le N^{-1} \|Ax\|_1 \le N^{-1/2} \|Ax\|_2 \le C \|x\|_2$.
The second inequality here is Cauchy-Schwartz, and the third one is the largest singular value bound.
Thierefore
\begin{equation}          \label{Kashin}
 C^{-1} \a(\d) \|y\|_2 \le N^{-1/2} \|y\|_1 \le  \|y\|_2
 \qquad \text{for all } y \in E.
\end{equation}
Subspaces $E$ possessing property \eqref{Kashin} are called
{\em Kashin's subspaces}. The classical Dvoretzky theorem states that a
high-dimensional Banach space has a subspace which is close to
Euclidean \cite{MS}. The dimension of such subspace depends on the
geometry of the ambient space. Milman proved that such subspaces
always exist in dimension $c\log n$, where $n$ is the dimension of
the ambient space \cite{Milman} (see also \cite{MS}). For the space
$\ell_1^n$ the situation is much better, and such subspaces exist in
dimension $(1-\d)n$ for any constant $\d>0$. This was first proved
by Kashin \cite{Kashin} also using a random matrix argument.
Obviously, as $\d \to 0$, the distance between the $\ell_1$ and
$\ell_2$ norms on such subspace grows to $\infty$. The optimal bound
for this distance has been found by Garnaev and Gluskin \cite{GG}
who used subspaces generated by Gaussian random matrices.

Kashin's subspaces turned out to be useful in theoretical computer science, 
in particular in the nearest neighbor search \cite{Indyk} and in compressed sensing.
At present no deterministic construction is known of such
subspaces of dimension $n$ proportional to $N$. The result
of \cite{R 06} shows that a $\lfloor (1+\d)n \rfloor \times n$
random Bernoulli matrix defines a Kashin's subspace with
$\a(\d)= c \d^2$. A random Bernoulli matrix is computationally
easier to implement than a random Gaussian matrix, while the
distance between the norms is not much worse than in the optimal
case. At the same time, since the subspaces generated by a Bernoulli
matrix are spanned by random vertices of the discrete cube, they
have relatively simple structure, which is possible to analyze.

\end{document}